\documentclass{amsart}
\usepackage{amsmath}
\usepackage{amsfonts}
\usepackage{amssymb}
\usepackage[utf8]{inputenc}
\usepackage{amssymb}
\usepackage{mathrsfs}
\usepackage{amsthm}
\usepackage{xcolor}
\usepackage{float}
\usepackage{tikz}
\usepackage{pgfplots}
\usepgfplotslibrary{fillbetween}

\usepackage{enumerate}

\usepackage{graphicx}

\newtheorem{ex}{Example}

\newtheorem{thrm}{Theorem}[section]
\newtheorem{cor}[thrm]{Corollary}
\newtheorem{lemma}[thrm]{Lemma}
\newtheorem{prop}[thrm]{Proposition}

\newtheorem{question}[thrm]{Question}
\newtheorem*{claim}{Claim}

\theoremstyle{definition}
\newtheorem{defin}[thrm]{Definition}
\newtheorem{rem}[thrm]{Remark}

\newtheorem*{xrem}{Remark}

\DeclareMathOperator{\Int}{Int}

\DeclareMathOperator{\supp}{supp}

\DeclareMathOperator{\rng}{rng}

\newcommand{\RR}{\mathbb{R}}

\newcommand{\la}{\langle}
\newcommand{\ra}{\rangle}

\begin{document}

\title{Some examples concerning $L\Sigma(\leq\omega)$ and metrizably fibered compacta}
\author{Antonio Avil\'es}
\address{Universidad de Murcia, Departamento de Matem\'{a}ticas, Campus de Espinardo 30100 Murcia, Spain.} \email{avileslo@um.es}
\author{Miko\l aj Krupski}
\address{Universidad de Murcia, Departamento de Matem\'{a}ticas, Campus de Espinardo 30100 Murcia, Spain\\ and \\ Institute of Mathematics\\ University of Warsaw\\ ul. Banacha 2\\
02--097 Warszawa, Poland }
\email{mkrupski@mimuw.edu.pl}

\begin{abstract}
The class of $L\Sigma(\leq\omega)$-spaces was introduced in 2006 by Kubi\'s, Okunev and Szeptycki as a natural refinement of the classical and important notion of Lindel\"of $\Sigma$-spaces. Compact $L\Sigma(\leq\omega)$-spaces were considered earlier, under different names, in the works of Tkachuk and Tkachenko in relation to metrizably fibered compacta. In this paper we give counterexamples to several open questions about compact $L\Sigma(\leq\omega)$-spaces that are scattered in the literature. Among other things, we refute a conjecture of Kubi\'s, Okunev and Szeptycki by constructing a separable Rosenthal compactum which is not an $L\Sigma(\leq\omega)$-space. We also give insight to the structure of first-countable $(K)L\Sigma(\leq\omega)$-compacta.
\end{abstract}



\maketitle

\section{Introduction}

All spaces under consideration are assumed to by Tychonoff.
Given a class $\mathcal{K}$ of compact spaces, Kubi\'s, Okunev and Szeptycki introduced in \cite{KOS} the following refinement of the classical notion of a Lindel\"of $\Sigma$-space: We say that a space $X$ is an \textit{$L\Sigma(\mathcal{K})$-space} if there is a separable metrizable space $M$ and
a compact-valued upper semicontinuous onto map $p:M\to X$ with $p(z)\in \mathcal{K}$, for all $z\in M$.
It is well known that $X$ is a Lindel\"of $\Sigma$-space if and only if $X$ is an $L\Sigma(\mathcal{K})$-space, where $\mathcal{K}$ is the class of all compact spaces (see \cite{Tkachuk1}).

In this paper we are mainly concerned with the case when $\mathcal{K}$ consists of compact metrizable spaces, i.e. compact spaces of countable weight. Following \cite{KOS}, we write $L\Sigma(\leq\omega)$ to denote this particular subclass of Lindel\"of $\Sigma$-spaces. More specifically, we say that $X$ is an \textit{$L\Sigma(\leq\omega)$-space} if $X$ is an $L\Sigma(\mathcal{K})$-space where $\mathcal{K}$ is the class of compact metrizable spaces.

Although the systematic study of $L\Sigma(\leq\omega)$-spaces was initiated in \cite{KOS}, compact $L\Sigma(\leq\omega)$-spaces were investigated earlier, independently by Tkachuk \cite{Tkachuk} (under the name \textit{weakly metrizably fibered spaces})  and Tkachenko \cite{Tkachenko} (under the name \textit{metrizably-approximable spaces}), as a natural generalization of so-called metrizably fibered spaces.

The aim of the present paper is to give counterexamples to several open question about compact $L\Sigma(\leq\omega)$-spaces scattered in the literature.
Let us describe our main results along with some motivations behind them.

\begin{ex}\label{Example A}
There is a separable Rosenthal compact space $KA$ which is not an $L\Sigma(\leq\omega)$-space.
\end{ex}

A compact space $K$ is a \textit{Rosenthal compact space} if $K$ is homeomorphic to a subspace of the space $B_1(X)$ of Baire class one functions on a Polish space $X$, equipped with the pointwise topology. It is was proved in \cite[Proposition 2.14]{KOS}  that if $K$ is homeomorphic to a subspaces of $B_1(X)$ consisting of functions with countably many discontinuities, then $K$ is an $L\Sigma(\leq\omega)$-space. This result motivated Kubi\'s, Okunev and Szeptycki to ask whether analogous assertion holds for all Rosenthal compacta (see \cite[Question 7.6]{KOS}, \cite[Problem 6]{O} or \cite[p. 24]{KM}). Example \ref{Example A} provides a negative answer to this question. Moreover, since the class of $L\Sigma(\leq\omega)$-spaces is stable under continuous images, the space $KA$ from Example A is a Rosenthal compactum which is not a continuous image of any compact subset of $B_1(\omega^\omega)$ consisting of functions with countably many discontinuities. As far as we know this is the first example of that sort (see \cite{AT} and Remark \ref{remark_CD} below).

We also give another counterexample to the question of Kubi\'s, Okunev and Szeptycki mentioned above. It is non-separable but has other interesting features. We need some notation first. For a set $\Gamma$ by $\Sigma(\mathbb{R}^\Gamma)$ we denote the following subset of the product $\RR^\Gamma$
$$\Sigma(\RR^\Gamma)=\{x\in \RR^\Gamma:|\{\gamma\in \Gamma:x(\gamma)\neq0\}|\leq \omega\}.$$
Put
$$\Sigma(\{0,1\}^\Gamma)=\Sigma(\RR^\Gamma)\cap \{0,1\}^\Gamma.$$

Recall that a compact space which, for some $\Gamma$, is homeomorphic to a subspace of $\Sigma(\RR^\Gamma)$ is called \textit{Corson compact}. A compact space which is homeomorphic to a weakly compact subset of a Banach space is called an \textit{Eberlein compact space}. A compact space $K$ is \textit{Gul'ko compact} if the space $C_p(K)$ of continuous functions on $K$ equipped with the pointwise topology, is a Lindel\"of $\Sigma$-space. It is well known that the class of Gul'ko compacta lies strictly between the class of Eberlein compacta and the class of Corson compact spaces.

\begin{ex}\label{Example B}
There is a compact space $KB\in B_1(\RR^2)\cap \Sigma(\{0,1\}^{\RR^2})$ such that $KB$ is not an $L\Sigma(\leq\omega)$-space yet it is an $L\Sigma(\mathcal{E})$-space, where $\mathcal{E}$ is the class of Eberlein compact spaces of cardinality not exceeding continuum.
\end{ex}

It was proved by Tkachuk in \cite{Tkachuk} that any Eberlein compact space of cardinality at most continuum is an $L\Sigma(\leq\omega)$-space. Later Molina Lara and Okunev \cite{MLO} generalized this to the class of Gul'ko compacta (see Section 5 below for a further generalization).
However, similar result for the class of Corson compact space is no longer true; a suitable example of a Corson compact space which is not $L\Sigma(\leq\omega)$ was given in \cite{KOS}. Example \ref{Example B} is a different space of that sort and has additional property of being Rosenthal. Moreover, since it is an
$L\Sigma(\mathcal{E})$-space, it can serve as a counterexample to Problems 4.10 and 4.11 in \cite{MLO}.

The next example provides a negative answer to Problems 4.12--4.14 in \cite{MLO}.

\begin{ex}\label{Example C}
There is a Corson compact $L\Sigma(\leq\omega)$-space which is not Gul'ko compact.
\end{ex}

Let $n$ be a positive integer. We say that a space $X$ is \textit{metrizably fibered} (\textit{$n$-fibered}) if there is a metrizable space $M$ and a continuous map $f:X\to M$ all of whose fibers $f^{-1}(z)$ are metrizable (have cardinality at most $n$).
Every metrizably fibered compact space $X$ is first-countable (see \cite{Tkachuk}). Let $D$ be the Alexandroff double circle space (see \cite[Example 3.1.26]{Eng}). Clearly, $D$ is $2$-fibered. By identifying all of the nonisolated points of $D$ we obtain a map of $D$ onto $A(\mathfrak{c})$, the one point compactification of a discrete set of size continuum.
Since $A(\mathfrak{c})$ is not first-countable, it is not metrizably fibered. Hence, the class of metrizably fibered compacta is not invariant under continuous images.
On the other hand, the class of compact $L\Sigma(\leq\omega)$-spaces is invariant under continuous images and every  metrizably fibered compact space is $L\Sigma(\leq\omega)$.
In view of the above Tkachuk asked in \cite{Tkachuk} the following two questions:

\begin{question}\cite[Problem 3.3]{Tkachuk}\label{question1}
Let $K$ be a metrizably fibered compact space. Is it true that every first-countable continuous image of $K$ is metrizably fibered?
\end{question}

\begin{question}\cite[Problem 3.4]{Tkachuk}\label{question2}
Is any first-countable compact $L\Sigma(\leq\omega)$-space a continuous image of a metrizably fibered space?
\end{question}
Regarding Question \ref{question1}, we prove that a lexicographic product of countably many intervals is a continuous image of a metrizably fibered compact space. In consequence the lexicographic product of three intervals can serve as a counterexample to Question \ref{question1}, as this space is not metrizably fibered (cf. \cite[Example 2.4]{Tkachuk}).
Actually, we have the following:

\begin{ex}\label{Example D}
 There is a compact first-countable space $KD$ which is not metrizable fibered yet $KD$ is a continuous image of a compact $3$-fibered space.
\end{ex}

It turns out however that we cannot replace 3-fibered by 2-fibered above. Namely, we shall prove (see Corollary \ref{theorem1.3} below):

\begin{thrm}
Let $K$ be a first-countable space. If $K$ is a continuous image of a 2-fibered compact space, then $K$ is metrizably fibered.
\end{thrm}

Regarding Question \ref{question2}, we obtain a partial solution given by the following:

\begin{ex}\label{Example E}
There is a compact first-countable $L\Sigma(\leq\omega)$-space $KE$ which is is not a continuous image of any compact metrizably fibered space.
\end{ex}


\section{Preliminaries}

In this section we collect basic definitions and facts that are used throughout the paper.

A \textit{set-valued map} from a space $X$ to a space $Y$ is a function that assigns to every point of $X$ a subset of $Y$. We say that a set-valued map $p:X\to Y$ is:
\begin{itemize}
 \item \textit{onto} if $\bigcup\{p(x):x\in X\}=Y$;
 \item \textit{compact-valued} if $p(x)$ is compact for all $x\in X$;
 \item \textit{upper semicontinuous} if for every open subset $U$ of $Y$, the set $\{x\in X:p(x)\subseteq U\}$ is open in $X$.
\end{itemize}

\begin{defin}
A space $X$ is an \textit{$L\Sigma(\leq\omega)$-space} if there is a separable metrizable space $M$ and
a compact-valued upper semicontinuous onto map $p:M\to X$ such that $p(z)$ is metrizable for all $z\in M$.
\end{defin}

The following fact is a part of folklore (cf. \cite[p. 2576]{KOS}):

\begin{prop}\label{LSigma-characterization}
Let $K$ be a compact space. The following two conditions are equivalent:
\begin{enumerate}[(A)]
 \item $K$ is an $L\Sigma(\leq\omega)$-space
 \item There is a countable cover $\mathcal{C}$ of $K$ consisting of closed subsets of $K$ such that for every $x\in K$ the intersection $\bigcap\{C\in \mathcal{C}:x\in C\}$ is metrizable.
\end{enumerate}
\end{prop}

\begin{xrem}
Spaces (not necessarily compact) satisfying condition $(B)$ above were first considered independently by Tkachuk in \cite{Tkachuk} and Tkachenko in
 \cite{Tkachenko}. In \cite{Tkachuk} they are called \textit{weakly metrizably fibered spaces}, whereas in \cite{Tkachenko} they are considered under the name \textit{metrizably-approximable spaces}. They were introduced and studied as a natural generalization of an important class of metrizably fibered spaces defined below.
\end{xrem}

\begin{defin}
 Let $n$ be a positive integer. A space $X$ is \textit{metrizably fibered} (\textit{$n$-fibered}) if there is a metrizable space $M$ and a continuous map $f:X\to M$ such that $f^{-1}(z)$ is metrizable (have cardinality at most $n$) for all $z\in M$.
\end{defin}

The following notion was introduced in \cite{KOS}:

\begin{defin}
 A space $X$ is a \textit{$KL\Sigma(\leq\omega)$-space} if there is a compact metrizable space $M$ and
a compact-valued upper semicontinuous onto map $p:M\to X$ such that $p(z)$ is metrizable for all $z\in M$.
\end{defin}

\begin{prop}\label{KLSigma_characterization}
 The following conditions are equivalent for any space $K$:
 \begin{enumerate}[(A)]
  \item $K$ is a $KL\Sigma(\leq\omega)$-space.
  \item $K$ is a continuous image of a compact metrizably fibered space.
  \item There is a family $\{C_s:s\in 2^{<\omega}\}$ of closed subsets of $K$ satisfying the following conditions:
\begin{enumerate}[(i)]
\item $C_\emptyset =K$
\item $C_s=C_{s\frown 0}\cup C_{s\frown 1}$, for every $s\in 2^{<\omega}$
\item for every $\sigma\in 2^\omega$ the set $\bigcap_{n}C_{\sigma|n}$ is metrizable.
\end{enumerate}
 \end{enumerate}
\end{prop}

\begin{proof}
Suppose that $K$ is a $KL\Sigma(\leq \omega)$-space. Fix a compact metrizable space $M$ and an upper semicontinuous onto map $p:M\to K$ so that $p(z)$ is metrizable for each $z\in M$. Let
$$L=\{\la z,x \ra\in M\times K:x\in p(z)\}$$
be the graph of $p$. Since $p$ is upper semicontinuous, the set $L$ is closed in $M\times K$ and hence it is compact. The projection onto the second coordinate maps $L$ onto $K$ (because $p$ is onto), whereas the projection onto the first coordinate maps $L$ onto a (metrizable) subspace of $M$ and its fibers are of the form $p(z)$, so they are metrizable. This proves $(A)\Rightarrow (B)$.

To prove $(B)\Rightarrow (A)$ fix a compact space $L$, a compact metrizable space $M$, a continuous surjection $f:L\to K$ and a continuous onto map $g:L\to M$ such that $g^{-1}(z)$ is metrizable for every $z\in M$. It is easy to check that the assignment
$$z\mapsto f(g^{-1}(z))$$
is an upper semicontinuous compact-valued map from $M$ onto $K$ and $f(g^{-1}(z))$ is metrizable being a continuous image of a compact metrizable space $g^{-1}(z)$.

To show $(A)\Rightarrow (C)$, fix a compact metric space $M$ and a compact valued upper semicontinuous map $p:M\to K$ such that $p(M)=K$ and $p(x)$ is metrizable, for all $x\in M$. The space $M$ is a continuous image of the Cantor set $2^\omega$. Hence, that there is a family $\{F_s:s\in 2^{<\omega}\}$ of closed subsets of $M$ such that $F_{\emptyset}=M$, $F_s=F_{s\frown 0}\cup F_{s\frown 1}$ and for every $\sigma\in 2^\omega$ the set $\bigcap_n F_{\sigma|n}$ is a singleton.
For $s\in 2^{<\omega}$, we define $$C_s=\bigcup\{p(x):x\in F_s\}.$$
Since
$p$ is compact-valued and upper semicontinuous, the set $C_s$ is compact, for every $s\in 2^{<\omega}$. Condition (i) follows from the fact that $p$ is onto, and condition (ii) is clear.
To show (iii), observe that if $\sigma\in 2^\omega$, then
$\bigcap_{n}F_{\sigma|n}=\{a\}$ for some $a\in M$
and $p(a)=\bigcap_{n} C_{\sigma|n}$ by compactness and upper semicontinuity of $p$. Indeed, for any open set $U$ in $K$, if $p(a)\subseteq U$, then the set $V=\{x\in M: p(x)\subseteq U\}$ is an open neighborhood of $a$ in $M$. Hence, $F_{\sigma|n}\subseteq V$, for all but finitely many $n$'s. This gives $C_{\sigma|n}\subseteq U$ for all but finitely many $n$'s so $p(a)=\bigcap_{n\in \omega} C_{\sigma|n}$. But $p(a)$ is metrizable by our assumption on $p$.

For $(C)\Rightarrow (A)$, fix a family $\{C_s:s\in 2^{<\omega}\}$ as in condition $(C)$. It can be easily verified that the multivalued map $p:2^\omega\to K$ given by
$p(\sigma)=\bigcap_{n}C_{\sigma|n}$ is is compact-valued upper semicontinuous and $p(\sigma)$ is metrizable, for every $\sigma\in 2^\omega$, by (iii).
\end{proof}

For a subset $A$ of a space $X$, we denote by $\chi_A:X\to\{0,1\}$ the characteristic function of the set $A$, given by the formula:
\begin{equation*}
\chi_A(x)=
  \left\{\begin{aligned}
  &   1 &\mbox{if }x\in A\\
&0  &\mbox{if }x\notin A
\end{aligned}
 \right.
\end{equation*}
It follows from the Baire criterion that if $X$ is a Polish space and $A\subseteq X$, then $\chi_A$ is a Baire class one function if and only if the set $A$ is both $G_\delta$ and $F_\sigma$.

\medskip

By $\mathbb{S}$ we will denote the \textit{split interval}, i.e. the space $((0,1]\times\{0\})\cup ([0,1)\times \{1\})$ endowed with the lexicographic order topology. To simplify notation, for $x\in [0,1]$, the points $\la x,0 \ra, \la x,1\ra\in \mathbb{S}$ will be denoted by $x^-$ and $x^+$ respectively.

\medskip

For a subset $A$ of a topological space we denote by $\overline{A}$ the closure of $A$. The interior of $A$ is denoted by $\Int A$.

\section{Example A: A Separable Rosenthal compact that is not an $L\Sigma(\leq\omega)$-space}

In this section we will construct a space $KA$ from Example \ref{Example A}.
For $a\in \RR^2$ let $D(a)$ be the open disk of radius 1 centered at $a$. By $\partial D(a)$ we denote the boundary (i.e. the circumference) of $D(a)$. For $a=\la a_1,a_2\ra\in \RR^2$ and $\theta\in \RR$ we set
\begin{align*}
&L_0(a,\theta)=D(a)\cup\{\la a_1+\cos\varphi,a_2+\sin\varphi\ra:\theta<\varphi<\theta+\pi\}\\
&L_1(a,\theta)=D(a)\cup\{\la a_1+\cos\varphi,a_2+\sin\varphi\ra:\theta<\varphi\leq\theta+\pi\}\\
&L_2(a,\theta)=D(a)\cup\{\la a_1+\cos\varphi,a_2+\sin\varphi\ra:\theta\leq \varphi<\theta+\pi\}\\
\end{align*}
Note that $L_i(a,\theta)$ is the open disk of radius 1 centered at $a$ with a half of its circumference attached (without endpoints if $i=0$ or with precisely one endpoint if $i=1,2$, see Figure \ref{figure1} below).


\begin{figure}[H]
\caption{Examples of $L_i(a,\theta)$ sets. For $i=0,1,2$, the set $L_i(a,\theta)$ consists of the shaded area together with the half of the circumference marked in black; The points $\alpha,\beta$ do not belong to $L_0(a,\theta)$; $\alpha\notin L_1(a,\theta)$, $\beta\in L_1(a,\theta)$; $\alpha\in L_2(a,\theta)$, $\beta\notin L_2(a,\theta)$.}
\label{figure1}

\medskip

\includegraphics{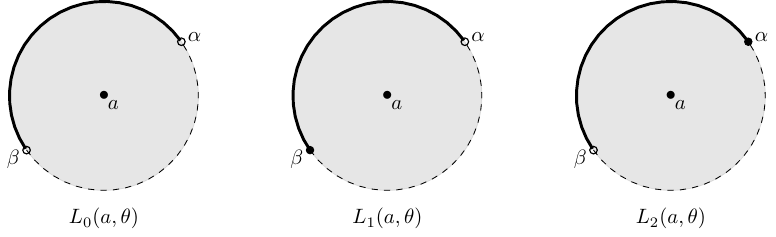}
\end{figure}


Let
$\mathcal{D}=\{D(p):p\in \mathbb{Q}^2\}$ and  let $$Y=\{\chi_D:D\in \mathcal{D}\}$$ be
considered as a subspace of the Cantor cube $\{0,1\}^{\mathbb{R}^2}$. Our compact space $KA$ is the closure of $Y$ in $\{0,1\}^{\mathbb{R}^2}$.

Let $$\mathcal{B}=\{L_i(a,\theta):i\in\{0,1,2\},\; a\in\RR^2,\; \theta\in \RR\}$$

We have the following:

\begin{lemma}\label{one inclusion}
$Y\cup\{\chi_B:B\in\mathcal{B}\}\cup\{\chi_\emptyset\}\subseteq KA$.
\end{lemma}
\begin{proof}
A basic clopen neighborhood of $\chi_\emptyset$ in $\{0,1\}^{\mathbb{R}^2}$ is of the form
$$C_F=\{\xi\in \{0,1\}^{\RR^2}:\xi(x)=0\mbox{ for }x\in F\},$$
where $F\subseteq \mathbb{R}^2$ is finite. Since for any finite $F\subseteq \RR^2$ we can find $D\in\mathcal{D}$ disjoint from $F$, the set $Y$ meets $C_F$ for each $F$. This gives $\chi_\emptyset\in KA$.

Fix $x=\la x_1,x_2\ra \in \RR^2$ and $\theta\in \RR$. Consider
the following two points in $\RR^2$:
$$\alpha=\la x_1+\cos\theta,x_2+\sin\theta\ra \mbox{ and } \beta=\la x_1+\cos(\theta+\pi), x_2+\sin(\theta+\pi)\ra$$
The points $\alpha$ and $\beta$ are the endpoints of the half of the circumference that is contained in each $L_i(x,\theta),\;i=0,1,2$ (see Figure \ref{figure1}).

\begin{claim}\label{claim}
$\chi_{L_0(x,\theta)}\in KA$
\end{claim}
\begin{proof}
Let $l$ be the line through $\alpha$ and $\beta$. The set $\RR^2\setminus l$ is the union of two disjoint open half-planes $H$ and $H'$ where $H$ contains the set $L_0(x,\theta)\setminus D(x)$.

To show that $\chi_{L_0(x,\theta)}\in KA$ let us consider the following subset of $\RR^2$ (cf. Figure \ref{figure2} below):
\begin{align*}
U=D(x)\cap H\cap\left(\RR^2\setminus \left(\overline{D(\alpha)}\cup \overline{D(\beta)} \right)\right)
\end{align*}

\begin{figure}[H]
\caption{The set $U$ consists of the shaded area without the boundary.}
\label{figure2}

\medskip

\includegraphics{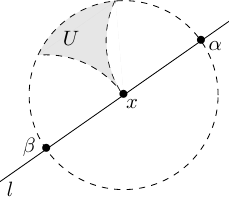}

\end{figure}


Clearly, $U$ is open in $\RR^2$ and $x\in \overline{U}$. So there is
a sequence $\{p_n\in \mathbb{Q}^2\cap U:n\in \omega\}$ that converges to $x$ in $\RR^2$.
We will check that
the sequence $\{\chi_{D(p_n)}:n\in\omega\}$ pointwise converges to $\chi_{L_0(x,\theta)}$.

Since $(p_n)_{n\in\omega}$ converges to $x$, if $z\in D(x)$ then $\chi_{D(p_n)}(z)=1$ for sufficiently large $n$. Similarly, if $z\notin \overline{D(x)}$ then eventually
$\chi_{D(p_n)}(z)=0$. It remains to verify that $(\chi_{D(p_n)}(z))_{n\in \omega}$ converges to $\chi_{L_0(x,\theta)}(z)$ for $z\in \partial D(x)$.

For every $n\in\omega$ denote
$$\{a_n, b_n\}=\partial D(p_n)\cap\partial D(x).$$
Let us assume that $a_n$ is closer to $\alpha$ than $b_n$, and $b_n$ is closer to $\beta$ than $a_n$.

Since $p_n\in U$, we have
\begin{align}
a_n\in H \mbox{ and }b_n\in H \label{eq2}
\end{align}
The points $a_n, b_n$ divide the circle $\partial D(x)$ into two open arcs $\Gamma_n$ and $\Delta_n$. One of them, say $\Gamma_n$, is entirely contained in $D(p_n)$ whereas $\Delta_n$ is disjoint from $D(p_n)$ (see Figure \ref{figure3} below).


\begin{figure}[H]
 \caption{The open arc $\Gamma_n$ marked in blue is contained in $D(p_n)$. The open arc $\Delta_n$ which is marked in red is disjoint from $D(p_n)$.}
 \label{figure3}

 \medskip

 \includegraphics{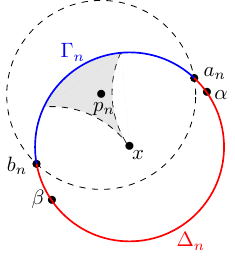}
\end{figure}


Since $(p_n)_{n\in\omega}$ converges to $x$, it follows that the sequences $(a_n)_{n\in\omega}$ and $(b_n)_{n\in\omega}$ converge to $\alpha$ and $\beta$, respectively.
So if $z\in  L_0(x,\theta)\cap \partial D(x)$, then $z\in \Gamma_n$ for sufficiently large $n$. Since $\Gamma_n\subseteq D(p_n)$, we infer that eventually $\chi_{D(p_n)}(z)=1=\chi_{L_0(x,\theta)}(z)$. If $z\in \partial D(x)\setminus L_0(x,\theta)$, then by \eqref{eq2}, $z\in \Delta_n$. As $\Delta_n\cap D(p_n)=\emptyset$ for all $n$, we conclude that $\chi_{D(p_n)}(z)=0=\chi_{L_0(x,\theta)}(z)$ for all $n$.
\end{proof}

By Claim, $$\chi_{L_0(x, \theta+1/n)},\;\chi_{L_0(x, \theta-1/n)} \in KA,\mbox{ for every } n\geq 1.$$
The sequences
$(\chi_{L_0(x,\theta+1/n)})_{n\in \omega}$ and  $(\chi_{L_0(x,\theta-1/n)})_{n\in \omega}$ pointwise converge to $\chi_{L_1(x,\theta)}$ and $\chi_{L_2(x,\theta)}$, respectively.
Hence, $\chi_{L_1(x,\theta)}, \chi_{L_2(x,\theta)}\in KA$.
\end{proof}

\begin{lemma}\label{Lemma-representation of K}
The space $KA$ is a separable Rosenthal compact space and
$KA=Y\cup\{\chi_B:B\in\mathcal{B}\}\cup\{\chi_\emptyset\}$.
\end{lemma}

\begin{proof}
By Lemma \ref{one inclusion} we only need to show that $KA\subseteq Y\cup \{\chi_B:B\in \mathcal{B}\}\cup\{\chi_\emptyset\}$.
The set $Y$ is countable and consists of Baire class one functions on $\RR^2$ (because each function in $Y$ is the characteristic function of an open disk which is simultaneously $F_\sigma$ and $G_\delta$ in $\RR^2$).
By Bourgain-Fremlin-Talagrand theorem (cf. \cite[p. 56]{Tod}), it is enough to check that every sequence of elements of $Y$ has a subsequence convergent to an element of $Y\cup \{\chi_B:B\in \mathcal{B}\}\cup\{\chi_\emptyset\}$.

To this end, take a nontrivial sequence $\{\chi_{D(p_n)}:p_n\in \mathbb{Q}^2\}$ of elements of $Y$. If $(p_n)_{n\in \omega}$ is unbounded, then
there is a subsequence $(p_{n_k})_{k\in\omega}$ with $\| p_{n_k} \|\to +\infty $, whence
$\chi_{D(p_{n_k})}$ converge to $\chi_{\emptyset}$ and we are done in this case. Suppose that $(p_n)_{n\in \omega}$ is bounded in $\RR^2$. By passing to a suitable subsequence, we may without loss of generality assume that $(p_n)_{n\in \omega}$ converges to a point $p\in \RR^2$. In particular, for sufficiently large $n$, the set $\partial D(p_n)\cap\partial D(p)=\{a_n,b_n\}$ for some distinct points $a_n,b_n\in\RR^2$. The points $a_n$ and $b_n$ partition the circle $\partial D(p)$ into two open arcs $\Gamma_n$ and $\Delta_n$. One of them, say $\Gamma_n$, is contained in $D(p_n)$ whereas $\Delta_n\cap D(p_n)=\emptyset$ (cf. Figure \ref{figure3}; now the point $p_n$ is not necessarily in $U$).

By compactness of $\partial D(p)$, there are subsequences $(a_{n_k})_{k\in \omega}$ and
$(b_{n_k})_{k\in \omega}$ convergent to some points $a\in \partial D(p)$ and $b\in \partial D(p)$, respectively. Since points $p_n$ converge to $p$, we must have $d(a,b)=2$, where $d(\cdot,\cdot)$ is the distance between points in $\RR^2$. At least one of the following three cases holds:

\medskip

\textit{Case 1:} The set $\{k\in \omega: a\in \Gamma_{n_k}\}$ is infinite. As $\Gamma_{n}\subseteq D(p_n)$ and $\Delta_n\cap D(p_n)$, it follows that $a\in D(p_{n_k})$ and $b\notin D(p_{n_k})$ for infinitely many $k$'s. Therefore, there is a subsequence of $\{D(p_{n_k}):k\in \omega\}$, convergent to
$L_1(p,\theta)$ or $L_2(p,\theta)$, for some $\theta$.

\textit{Case 2:} The set $\{k\in \omega: b\in \Gamma_{n_k}\}$ is infinite. This case is analogous to the previous one.

\textit{Case 3:} Both $\{k\in \omega: a\in \Gamma_{n_k}\}$ and $\{k\in \omega: b\in \Gamma_{n_k}\}$ are finite. It is easy to see that in this case the sequence $\{D(p_{n_k}):k\in \omega\}$ converges to $L_0(p,\theta)$, for some $\theta$.
\end{proof}

\begin{thrm}\label{KA main result}
The space $KA$ is not an $L\Sigma(\leq\omega)$-space.
\end{thrm}
\begin{proof}
Let $\mathbb{R}^2\cup\{\infty\}$ be the one point compactification of $\RR^2$.
Consider the function $s:\mathcal{D}\cup\mathcal{B}\to \mathbb{R}^2$ that assigns to each disk its center. Identifying a set with its characteristic function and declaring that $s(\emptyset)=\infty$, we may view $s$ as a map $s:KA\to \mathbb{R}^2\cup\{\infty\}$ (cf. Lemma \ref{Lemma-representation of K}). Observe that $s$ is sequentially continuous. True, if a sequence $(\xi_n)_{n\in\omega}$ of elements of $KA$ converges to $\chi_\emptyset$ then the sequence $(s(\xi_n))_{\omega}$ has no convergent subsequence so it converges to $\infty$. If $(\xi_n)_{n\in\omega}$ converges to $\xi\in KA\setminus\{\chi_\emptyset\}$, then the sequence $(s(\xi_n))_{\omega}$ must converge
to the center of the disk defining $\xi$. Since $s$ is sequentially continuous, it is continuous by Bourgain-Fremlin-Talagrand theorem.

We will show that condition $(B)$ of Proposition \ref{LSigma-characterization} fails. To this end,
fix a countable compact cover $\mathcal{C}=\{C_n:n\in \omega\}$ of $KA$.
Since $s$ is continuous, the set
$$A_n=s(C_n)\cap\RR^2$$
is closed in $\mathbb{R}^2$, for all $n\in \omega$.
By the Baire Category Theorem there is
$$p=\la p_1,p_2 \ra\in \RR^2\setminus \bigcup_{n\in\omega}\left(A_n\setminus \Int A_n \right).$$
We will show that if $\xi\in s^{-1}(p)$, then the set $L=\bigcap\{C_n\in \mathcal{C}:\xi\in C_n\}$ contains a copy of the split interval $\mathbb{S}$ and hence, it is nonmetrizable.

Consider
$$K_p=s^{-1}(p)\cap\{\chi_B:B\in \mathcal{B}\},$$
i.e. $K_p$ consists of all elements of $KA$ corresponding to non-open disks centered at $p$.
Fix $\xi\in s^{-1}(p)$ and let $$L=\bigcap\{C_n\in \mathcal{C}:\xi\in C_n\}.$$
Note that $\xi\in C_n$ implies that $s(\xi)=p\in A_n$. By the choice of $p$ this means that $p\in \Int A_n$.

\begin{claim}
 We have $K_p\subseteq C_n$ for all $n$ with $\xi\in C_n$.
\end{claim}
\begin{proof}
Let us first show that $\chi_{L_0(p,\theta)}\in C_n$ for all $\theta$. To this end, fix $\theta$ and
consider the following point $m\in \mathbb{R}^2$:
$$m=\la p_1+\cos(\theta+\pi/2),p_2+\sin(\theta+\pi/2)\ra.$$
Note that $m$ is the midpoint of the arc lying on the boundary of $L_0(p,\theta)$. Let
$I$ be the open interval in $\RR^2$ joining $m$ and $p$. It follows from $p\in \Int A_n$ that there is a
sequence $\{a_k\in I\cap A_n:k\in\omega\}$ convergent to $p$. Since $a_k\in A_n$, for each $k$ we may find $\xi_k\in C_n$ so that $s(\xi_k)=a_k$. It is easy to see that
$(\xi_k)_{k\in\omega}$ pointwise converges to $\chi_{L_0(p,\theta)}$. Therefore
$\chi_{L_0(p,\theta)}\in C_n$ because $C_n$ is closed.

The sequences
$(\chi_{L_0(p,\theta+1/i)})_{n\in \omega}$ and  $(\chi_{L_0(p,\theta-1/i)})_{i\in \omega}$ pointwise converge to $\chi_{L_1(p,\theta)}$ and $\chi_{L_2(p,\theta)}$, respectively.
As we have already proved
$$\chi_{L_0(p,\theta+1/i)},\; \chi_{L_0(p,\theta-1/i)}\in C_n.$$
Hence
$\chi_{L_1(p,\theta)}, \chi_{L_2(p,\theta)}\in C_n$.
\end{proof}

Define a map $f:\mathbb{S}\to K_p$ by letting
$$f(x^-)=\chi_{L_2(p,x)},\quad f(x^+)=\chi_{L_1(p,x)}.$$
It can be readily checked that $f$ is
one-to-one and continuous so it is an embedding. By Claim, $K_p\subseteq L$ so $L$ contains a copy of $\mathbb{S}$.
\end{proof}

\begin{rem}\label{remark_CD}
We may consider the following two subclasses of the class $\mathcal{R}$ of all Rosenthal compacta. By $\mathcal{RK}$ denote the class of compact spaces that are homeomorphic to a subspace of $B_1(C)$ of Baire class one functions on the Cantor set $C$. A compact space $K$ belongs to the class
$\mathcal{CD}$ if $K$ is homeomorphic to a compact subset of $\mathbb{R}^X$ consisting of functions with countably many discontinuities, for some Polish space $X$. It is well known that $\mathcal{CD}\subseteq \mathcal{RK}\subseteq\mathcal{R}$ (see \cite{MP}). The first example of a Rosenthal compact space $K$ which is not in $\mathcal{RK}$ was given by Pol in \cite{Pol}. However only very recently the first example of a space distinguishing $\mathcal{CD}$ and $\mathcal{RK}$ was found by the first author and Todorcevic \cite{AT}.

Note that $KA$ may be viewed as a compact subspace of $B_1(\mathbb{R}^2\cup\{\infty\})$ of Baire class one functions on the one point compactifications of the plane $\RR^2$.  Thus, $KA\in \mathcal{RK}$.
If $K\in \mathcal{CD}$, then $K$ is an $L\Sigma(\leq\omega)$-space (see \cite[Proposition 2.14]{KOS}). Hence, by Theorem \ref{KA main result}, we get $KA\notin \mathcal{CD}$. Moreover, since the class of $L\Sigma(\leq\omega)$-spaces is stable under taking continuous images \cite[Proposition 2.4]{KOS}, the space $KA$ is not a continuous image of any space $K\in \mathcal{CD}$.
\end{rem}

\section{Example B: A countably supported Rosenthal compact that is not an $L\Sigma(\leq\omega)$-space}

In this section, we will construct a space $KB$ from Example B. The space $KB$ which we are going to define will be a compact subspace of $B_1(\RR^2)\cap\Sigma(\{0,1\}^{\RR^2})$.

Let $\mathcal{A}$ be the collection of all subsets $A$ of $\RR^2$ that simultaneously satisfy the following two conditions:
\begin{align}
&\mbox{If } \la x_1,y_1\ra,\ldots , \la x_n,y_n\ra\in A \mbox{ and } x_1<\ldots < x_n\mbox{ then } \forall{i}\;\;
|x_{i+1}-x_i|\leq 2^{-i} \label{condition 1}.\\
&\mbox{For every } x\in \RR\;\; \mbox{the set } (\{x\}\times \RR)\cap A \mbox{ is at most 1-element}. \label{condition 2}
\end{align}
Let $KB=\{\chi_A:A\in \mathcal{A}\}$ be considered as a subspace of the Cantor cube $\{0,1\}^{\mathbb{R}^2}$. Since the definition of $KB$ is of finite character, $KB$ is compact. Let $\pi:\RR^2\to \RR$ be the projection onto the first coordinate. Note that condition \eqref{condition 1} implies that for every $A\in \mathcal{A}$, the set $\pi(A)$ has at most one accumulation point. Hence, by \eqref{condition 2}, the set $A$ is countable and has at most one non-isolated point. It follows that all elements of $KB$ are of the first Baire class and belong to the $\Sigma$-product $\Sigma(\{0,1\}^{\RR^2})$.

In the remaining part of this section, we will use the following notation.
For an infinite $A\in \mathcal{A}$, by $\alpha(A)$ we will denote the unique accumulation point of $\pi(A)$. For $t\in \RR^2$, the map $p_t:KB\to \{0,1\}$ is the projection onto coordinate $t$.

\begin{thrm}\label{thrm:ExampleB}
The space $KB$ is not an $L\Sigma(\leq\omega)$-space.
\end{thrm}
\begin{proof}
Striving for a contradiction, 
suppose that $\mathcal{C}=\{C_1,C_2,\ldots\}$ is a countable closed cover of $KB$ such that for any $\xi\in KB$, the set
$$C_{\xi}=\bigcap\{C\in \mathcal{C}:\xi\in C\}$$
is metrizable (cf. Proposition \ref{LSigma-characterization}). Consider the family $\mathcal{A}_0\subseteq \mathcal{A}$ that consists of
all infinite elements $A$ of $\mathcal{A}$ that satisfy the following conditions:
\begin{align}
&\alpha(A)\notin \pi(A). \\
&\mbox{If } \la x_1,y_1 \ra,\dots , \la x_n,y_n\ra\in A \mbox{ and } x_1<\dots <x_n, \mbox{ then } |\alpha(A)-x_i|<2^{-i}. \label{tame}
\end{align}

Take $A\in \mathcal{A}_0$.
Clearly, for any $y\in \RR$, the set
$$A_y=A\cup\{\la \alpha(A),y\ra\}$$
belongs to $\mathcal{A}$.
Let
$$K(A)=\{\chi_A\}\cup\{\chi_{A_y}:y\in \RR\}.$$
It is easy to see that $K(A)$, considered as a subspace of $\{0,1\}^{\RR^2}$, is a copy of the one point compactification of a discrete set of size $\mathfrak{c}$. Hence, the set
$$\{y\in \RR:\chi_{A_y}\in C_{\chi_A}\}$$
must be countable, as $C_{\chi_A}$ is metrizable. In particular,
for every $A\in \mathcal{A}_0$ and every finite set $F\subseteq \RR$, we have
\begin{equation}\label{eq:induction}
 \{n:\chi_A\in C_n \mbox{ and }\chi_{A_y} \notin C_n \mbox{ for some }y\in \RR\setminus F\}\neq\emptyset.
\end{equation}
Let
$$n_1=\min\{n: \exists A\in \mathcal{A}_0\;\; \chi_A\in C_n\mbox{ and }\chi_{A_y} \notin C_n \mbox{ for some }y\in \RR\}.$$
Fix $X^0\in \mathcal{A}_0$ and $a(0)\in \RR$ such that $\chi_{X^0}\in C_{n_1}$ and $\chi_{X^0_{a(0)}}\notin C_{n_1}$.
Let $U_1$ be a basic clopen neighborhood of $\chi_{X^0_{a(0)}}$ disjoint from $C_{n_1}$. The set $U_1$ can be taken of the form
$$U_1=\bigcap\{p^{-1}_t(0):t\in F^1_0\}\cap \bigcap\{p^{-1}_t(1):t\in F^1_1\},$$
for some finite sets $F^1_0, F^1_1\subseteq \RR^2$.
Write $F^1_1=\{\la x^1_1,y^1_1\ra,\dots , \la x^1_{m(1)},y^1_{m(1)}\ra \}$, where
$x^1_1<\dots <x^1_{m(1)}$ and $x^1_{m(1)}=\alpha(X^0),\; y^1_{m(1)}=a(0)$.
Let
$$\gamma_1=\min\{x^1_i+2^{-i}:i=1,\dots ,m(1)\}.$$
Since $X^0\in \mathcal{A}_0$ and $F^1_1\subseteq X^0_{a(0)}$, by definition of $\gamma_1$ we get $\gamma_1>\alpha(X^0)$ (see \eqref{tame}). Also, we infer from \eqref{tame} that the family
$$\mathcal{A}_1=\{A\in \mathcal{A}_0:\chi_A\in U_1 \mbox{ and } \alpha(A)<\gamma_1\},$$
is nonempty.

Inductively, we construct:
\begin{itemize}
\item an increasing sequence $n_1<n_2<\dots$ of positive integers,
\item increasing sequences of finite sets $F^1_0\subseteq F^2_0\subseteq \dots$ and $F^1_1\subseteq F^2_1\subseteq\dots$
\item a sequence $\gamma_1\geq \gamma_2\geq\dots$ of real numbers and
\item a decreasing sequence of nonempty families $\mathcal{A}_0\supseteq \mathcal{A}_1\supseteq \mathcal{A}_2\supseteq\dots$
\end{itemize}
such that
if
$$F_1^k=\{\la x_1,y_1\ra,\dots , \la x_{m(k)},y_{m(k)}\ra\}\mbox{, where }
x_1<\dots <x_{m(k)} \mbox{ and } m(1)<m(2)<\ldots $$ and if
$$U_k=\bigcap\{p^{-1}_t(0):t\in F^k_0\}\cap \bigcap\{p^{-1}_t(1):t\in F^k_1\},$$ then
\begin{enumerate}[(i)]
\item $\gamma_k=\min\{x_i+2^{-i}:i=1,\dots ,m(k)\}$,
\item $\mathcal{A}_k=\{A\in \mathcal{A}_0:\chi_{A}\in U_k \mbox{ and }\alpha(A)<\gamma_k\}$,
\item $U_k\cap C_{n_k}=\emptyset$
\item $n_{k+1}=\min\{n:\exists A\in \mathcal{A}_k\;\;\chi_A\in C_n \mbox{ and } \chi_{A_y}\notin C_n \mbox{ for some }y\in \RR\}.$
\item $x_{m(k+1)}<\gamma_k$
\end{enumerate}

Fix $k\geq 1$ and suppose that $n_i\;, F^i_0,\; F^i_1,\; \gamma_i$ and $\mathcal{A}_i$ are constructed for all $i\leq k$.
Let $F$ be the projection of the set $F^k_0$ onto the second coordinate. Define (cf. \eqref{eq:induction})
$$n_{k+1}=\min\{n: \exists A\in \mathcal{A}_k\;\; \chi_A\in C_n\mbox{ and }\chi_{A_y} \notin C_n \mbox{ for some }y\in \RR\setminus F\}.$$
Since $\mathcal{A}_k\subseteq \mathcal{A}_{k-1}$, we infer from the inductive assumption on $n_k$ (cf. (iii)) that $n_k<n_{k+1}$.
Pick $X^{k}\in \mathcal{A}_k$ and $a(k)\in \RR\setminus F$ with $\chi_{X^{k}}\in C_{n_{k+1}}$ and $\chi_{X^{k}_{a(k)}}\notin C_{n_{k+1}}$. Note that $X^k\in \mathcal{A}_k$ and $a(k)\notin F$ imply that $\chi_{X^{k}_{a(k)}}\in U_k$.

We take $U_{k+1}\subseteq U_k$ a basic clopen neighborhood of $\chi_{X^k_{a(k)}}$ disjoint from $C_{n_{k+1}}$ of the form
$$U_{k+1}=\bigcap\{p^{-1}_t(0):t\in F^{k+1}_0\}\cap \bigcap\{p^{-1}_t(1):t\in F^{k+1}_1\},$$
for some finite sets $F^{k+1}_0, F^{k+1}_1\subseteq \RR^2$ containing $F^k_0$ and $F^k_1$, respectively. Write $F^{k+1}_1=\{\la x_1,y_1\ra ,\dots , \la x_{m(k+1)},y_{m(k+1)}\ra\}$, where $m(k+1)>m(k)$,
$x_1<\dots <x_{m(k+1)}=\alpha(X^k)$ and $y_{m(k+1)}=a(k)$. We set
$$\gamma_{k+1}=\min\{x_i+2^{-i}:i=1,\dots ,m(k+1)\}.$$
Since $F^k_1\subseteq F^{k+1}_1$, we get $\gamma_{k+1}\leq \gamma_k$, by (i).
Similarly as for $\gamma_1$, we argue that
$\gamma_{k+1}>\alpha(X^k)$ and that the family
$$\mathcal{A}_{k+1}=\{A\in \mathcal{A}_k:\chi_A\in U_{k+1} \mbox{ and } \alpha(A)<\gamma_{k+1}\},$$
is nonempty. This finishes the inductive construction.

Let
$$S=\bigcup_{k=1}^\infty F^k_1.$$
The set $S$ is infinite because for each $k$, we have $(\alpha(X^k),a(k))\in F^{k+1}_1\setminus F^k_1$. Moreover, from (v) and $\gamma_1\geq\gamma_2\geq\dots$, we get $S\in \mathcal{A}_k$ for all $k=0,1,\dots$.

\begin{claim}
$K(S)\subseteq \bigcap\{C\in \mathcal{C}:\chi_{S}\in C\}$.
\end{claim}
\begin{proof}
Pick $z\in \RR$. Aiming at a contradiction, suppose $\chi_{S}\in C_n$ and $\chi_{S_z}\notin C_n$, for some $n$. Since $n_1<n_2<\ldots$, there exists $k$ such that $n<n_{k+1}$. But $S\in \mathcal{A}_k$, so this contradicts the choice of $n_{k+1}$ (see (iv)).
\end{proof}
Since $K(S)$ is nonmetrizable being a copy of the one point compactification of an uncountable discrete set, it follows that $\bigcap\{C\in \mathcal{C}:\chi_{S}\in C\}$ is nonmetrizable, contradicting the assumption on the cover $\mathcal{C}$.
\end{proof}

Denote by $\mathcal{E}$ the class of Eberlein compacta of cardinality not exceeding continuum. We say that a space $X$ is an \textit{$L\Sigma(\mathcal{E})$-space} if there is a separable metrizable space $M$ and
a compact-valued upper semicontinuous onto map $p:M\to X$ such that $p(z)\in \mathcal{E}$ for all $z\in M$. If $X$ is compact this is equivalent to saying that there is a countable closed cover $\mathcal{C}$ of $X$ such that for every $x\in X$ the intersection $\bigcap\{C\in \mathcal{C}:x\in C\}$ is Eberlein of cardinality not exceeding continuum (cf. Proposition \ref{LSigma-characterization}).

\begin{prop}\label{proposition:LSigma(Eberlein)}
The space $KB$ is an $L\Sigma(\mathcal{E})$-space.
\end{prop}
\begin{proof}
Let $\mathcal{B}$ be a countable basis for $\RR^2$.
For every $U\in \mathcal{B}$, let
$$C_U=\{\xi\in KB:\forall a\in U\;\; \xi(a)=0\}.$$
It is clear that $C_U$ is closed in $KB$. So the family
$$\mathcal{C}=\{C_U:U\in \mathcal{B}\},$$
is a countable closed cover of $KB$.

Fix $Y\in \mathcal{A}$.
Observe that if $\chi_A\in \bigcap\{C\in \mathcal{C}:\chi_Y\in C\}$, then $\overline{A}\subseteq\overline{Y}$ (the closures are taken in $\mathbb{R}^2$).
So if $Y$ is finite, then $\bigcap\{C\in \mathcal{C}:\chi_Y\in \mathcal{C}\}$ is finite.
Suppose that $Y$ is infinite and let $\alpha(Y)$ be the accumulation point of $\pi(Y)$. Let $Y'=Y\setminus ((\alpha(Y)\times \RR)$.
We have
$\overline{Y}\subseteq Y\cup (\alpha(Y)\times \RR)$ so restricting coordinates we get $\bigcap\{C\in \mathcal{C}:\chi_Y\in C\}\subseteq\{0,1\}^{Y'\cup (\alpha(Y)\times \RR)}$ and using condition \eqref{condition 2} we infer that $\bigcap\{C\in \mathcal{C}:\chi_Y\in C\}$ embeds into the product $$\{0,1\}^{Y'}\times \sigma_1\left(\{0,1\}^{\alpha(Y)\times \RR}\right),$$ where $\sigma_1\left(\{0,1\}^{\alpha(Y)\times \RR}\right)$ is the subset of $\{0,1\}^{\alpha(Y)\times \RR}$ consisting of elements with at most one nonzero coordinate.
Since $\sigma_1\left(\{0,1\}^{\alpha(Y)\times \RR}\right)$ is homeomorphic to the one point compactification of the discrete set of size $\mathfrak{c}$ and $Y$ is countable, the space $\{0,1\}^{Y'}\times \sigma_1\left(\{0,1\}^{\alpha(Y)\times \RR}\right)$ is an Eberlein compact space of cardinality continuum.
\end{proof}

It is known (see \cite[Theorem 2.15]{Tkachuk}) that if $K\in \mathcal{E}$, then $K$ is an $L\Sigma(\leq\omega)$-space. Hence we have the following corollary to Theorem \ref{thrm:ExampleB} and Proposition \ref{proposition:LSigma(Eberlein)} which answers Problems 4.10 and 4.11 from \cite{MLO} in the negative:

\begin{cor}
The space $KB$ is an $L\Sigma(L\Sigma(\leq\omega))$-space but not an $L\Sigma(\leq\omega)$-space. In particular, the classes of $L\Sigma(L\Sigma(\leq\omega))$-spaces and of $L\Sigma(\leq\omega)$-spaces are different in the realm of compacta.
\end{cor}

\section{Example C and Corson compacta that are $L\Sigma(\leq\omega)$-spaces}

Given a set $\Gamma$ and a point $x\in \mathbb{R}^\Gamma$, \textit{the support} of $x$ is the set
$$\supp x=\{\gamma\in \Gamma:x(\gamma)\neq 0\}.$$
The following subclass of Corson compacta was introduced by A. Leiderman \cite{L} under the name \textit{almost Gul'ko compact spaces}. Our notation follows Todorcevic \cite{Tod1}. We refer the interested reader to \cite{Tod1} for a more general notion of $\mathcal{E}_2(\theta)$-spaces.

\begin{defin}\label{def:E_2}
We say that a compact subspace $K$ of some $\Sigma$-product $\Sigma(\mathbb{R}^\Gamma)$ has the property $\mathcal{E}_2(\aleph_1)$ if there is a sequence $(\Gamma_n)_{n\in \omega}$ of subsets of $\Gamma$ such that if, for $x\in K$, we let $N_x=\{n\in \omega: |\supp x\cap \Gamma_n|<\aleph_0\}$, then the set $\Gamma\setminus\bigcup_{n\in N_x} \Gamma_n$ is countable.
\end{defin}

We will give an elementary proof the following

\begin{thrm}\label{E_2 are LSigma}
If a compact space $K$ of cardinality $\leq\mathfrak{c}$ has the property
$\mathcal{E}_2(\aleph_1)$, then $K$ is an $L\Sigma(\leq\omega)$-space.
\end{thrm}

Sokolov showed in \cite{S} that every Gul'ko compact space has the property $\mathcal{E}_2(\aleph_1)$. It is also known that there are compact spaces of cardinality $\mathfrak{c}$ which are not Gul'ko compact but have the property $\mathcal{E}_2(\aleph_1)$ (see \cite{L1} or \cite{Tod1}). Therefore, Theorem \ref{E_2 are LSigma} generalizes \cite[Theorem 4.4]{MLO} and provides a space for Example \ref{Example C}.
It is worth mentioning that another extension of \cite[Theorem 4.4]{MLO} (different than ours), using a similar concept as in Definition \ref{def:E_2}, was given by Rojas-Hern\'andez in \cite{RH}.

\begin{proof}[Proof of Theorem \ref{E_2 are LSigma}]
Fix a set $\Gamma$ such that $K\subseteq \Sigma(\RR^\Gamma)$ and fix a sequence
$(\Gamma_n)_{n\in\omega}$ of subsets of $\Gamma$ as in Definition \ref{def:E_2}.
Consider
$$T=\bigcup\{\supp x:x\in K\}.$$
Since $K\subseteq \Sigma(\RR^\Gamma)$ satisfies $|K|\leq \mathfrak{c}$, the set $T$ has cardinality at most $\mathfrak{c}$. The restriction mapping $x\mapsto x\upharpoonright T$ embeds $K$ into the $\Sigma$-product $\Sigma(\RR^T)$ and if we let $T_n=\Gamma_n\cap T$, then the set $T$ and the sequence $(T_n)_{n\in\omega}$ have the property from Definition \ref{def:E_2}. Denote by $I$ the unit interval $[0,1]$ and fix an arbitrary injection $h:T\to I$. Put $I_{n+1}=h(T_n)$ and $I_0=I\setminus\rng(h)$, where $\rng(h)$ is the range of $h$. It is clear that $h$ defines a homeomorphic embedding of $K$ into the $\Sigma$-product $\Sigma(\RR^I)$ and if, for $x\in K$, we let $N_x=\{n\in\omega:|\supp x\cap I_n|<\aleph_0\}$, then the set $I\setminus\bigcup_{n\in N_x} I_n$ is countable.

For rationals $p<q$ and integer $n$ define
$$C_{n,p,q}=\{x\in K:\forall\gamma\in (p,q)\cap I_n\;\;x(\gamma)=0\}.$$
It can be easily checked that the set $C_{n,p,q}$ is closed in $K$.
If $z\in K$, then $I\setminus \bigcup\{I_n:n\in N_z\}$ is countable. Hence, for some $m\in N_z$ the set $I_m$ is uncountable. Since $\supp z\cap I_m$ is finite, there are rationals $r<t$ such that $z(\gamma)=0$ for every $\gamma\in (r,t)\cap I_m$. This shows that $z\in C_{m,r,t}$ and thus the family $\{C_{n,p,q}: p,q\in \mathbb{Q},\;n\in \omega\}$ is a countable closed cover of $K$.

Fix $z\in K$ and let $C_z=\bigcap\{C_{n,p,q}:z\in C_{n,p,q}\}$. The set $$R_z=I\setminus\left(\bigcup_{n\in N_z}I_n\right)$$ is countable. We claim that
\begin{equation}\label{equation}
\mbox{if } x\in C_z\mbox{, then } \supp x\subseteq \supp z\cup R_z. \end{equation}

Indeed, otherwise $x(\gamma)\neq 0$ and $z(\gamma)=0$ for some $\gamma\in \bigcup_{n\in N_z}I_n$. Take $n\in N_z$ with $\gamma\in I_n$. Since $n\in N_z$, the set $$F=\supp z\cap I_n$$
is finite and $\gamma\notin F$. It follows that there are rationals $p<q$ such that $$\gamma\in (p,q) \mbox{ and } (p,q)\cap F=\emptyset.$$
This gives $z\in C_{n,p,q}$ and thus $x\in C_{n,p,q}$ (because $x\in C_z$). This contradicts $x(\gamma)\neq 0$.

Since the set $\supp z\cup R_z$ is countable, we infer from \eqref{equation} that the set $C_z$ is metrizable.
\end{proof}

\section{Example D: A first countable $KL\Sigma(\leq\omega)$-space need not be metrizably fibered}

Tkachuk asked in \cite{Tkachuk} whether every first-countable continuous image of a metrizably fibered compactum is metrizably fibered (see Question \ref{question1}). A related question of Tkachuk, whether every first-countable continuous image of the lexicographic square is metrizably fibered, was answered in the affirmative by Daniel and Kennaugh \cite{DK} (see \cite{D} for a slight generalization). However, as we we will show in this section, in general the answer to Question \ref{question1} is in the negative (see Corollary \ref{cor_lex_product}).

The following fact is easy do derive.

\begin{prop}\label{KLSigma_characterization_2}
Let $K$ be a compact space. The following conditions are equivalent:
\begin{enumerate}[(A)]
 \item $K\in KL\Sigma(\leq \omega)$
 \item There is a compact metric space $M$ and a closed subspace $F$ of $K\times M$ such that $\pi_1(F)=K$ and $\pi_2^{-1}(y)\cap F$ is metrizable for every $y\in M$, where $\pi_1:K\times M\to K$ and $\pi_2:K\times M\to M$ are the projections.
\end{enumerate}
\end{prop}
\begin{proof}
 The implication $(B)\Rightarrow (A)$ follows from \cite[Proposition 2.3]{KOS}. To show the converse,
 suppose that $K\in KL\Sigma(\leq\omega)$. By definition, there is a compact metric space $M$ and a compact-valued upper semicontinuous map $p:M\to K$ such that $K=\bigcup\{p(y):y\in M\}$ and $p(y)$ is metrizable for all $y\in M$. Let
 $$F=\{\la x,y \ra \in K\times M:x\in p(y)\}.$$
 The set $F$ is closed being the graph of the upper semicontinuous function $p$. It is easy to verify that $F$ is as required.
\end{proof}

For an ordinal number $\lambda$, by $[0,1]^\lambda_{lex}$ we denote the space $[0,1]^\lambda$ endowed with the order topology given by the lexicographic product order. We write $[0,1]^\lambda$ when the usual product topology on $[0,1]^\lambda$ is considered. Let us show the following:

\begin{thrm}\label{ExampleD:main}
 For every countable ordinal $\lambda<\omega_1$, the compact space $[0,1]^\lambda_{lex}$ is a $KL\Sigma(\leq\omega)$-space.
\end{thrm}
\begin{proof}
 Consider the following subset $F$ of $[0,1]^\lambda_{lex}\times [0,1]^\lambda$:
 $$F=\{\la (x_\alpha), (y_\alpha)  \ra: (\exists \beta)\; \left(\forall\alpha<\beta\;\;x_\alpha=y_\alpha\right) \mbox{ and } \left(\forall\alpha\geq \beta\; x_\alpha=0 \mbox{ or } \forall\alpha\geq \beta\; x_\alpha=1 \right)\}.$$

 \begin{claim}
 The set $F$ is closed in $[0,1]^\lambda_{lex}\times [0,1]^\lambda$.
 \end{claim}
\begin{proof}
Pick $\la(x_\alpha), (y_\alpha) \ra\in
[0,1]^\lambda_{lex}\times [0,1]^\lambda$ so that $\la(x_\alpha), (y_\alpha) \ra\notin F$. Let $$\beta=\min\{\alpha<\lambda: x_\alpha\neq y_\alpha\}.$$
Let $A$ and $B$ be two disjoint open subsets of $[0,1]$ satisfying $x_\beta\in A$ and $y_\beta\in B$. Denote by $p_\beta:[0,1]^\lambda\to [0,1]$ the projection onto the coordinate $\beta$. Consider the following three cases:

\medskip

\textit{Case 1:} $x_\beta\notin\{0,1\}$. Consider the following set:
$$W=\{(z_\alpha)\in [0,1]^\lambda_{lex}:z_\alpha=x_\alpha\mbox{ for }\alpha<\beta\mbox{ and }z_\beta\in A\cap (0,1)\}.$$
It is easy to see that $W$ is open in $[0,1]^\lambda_{lex}$ and $(x_\alpha)\in W$. Hence, the set $W\times p_\beta^{-1}(B)$ is an open neighborhood of $\la(x_\alpha), (y_\alpha) \ra$ in $[0,1]^\lambda_{lex}\times [0,1]^\lambda$ disjoint from $F$.

\medskip

\textit{Case 2:} $x_\beta=0$. Since $\la(x_\alpha), (y_\alpha) \ra\notin F$, there is $\alpha>\beta$ with $x_\alpha\neq 0$. Let
$$\gamma=\min\{\alpha>\beta:x_\alpha\neq 0\}.$$
Fix $\varepsilon>0$ such that $[0,\varepsilon)\subseteq A$. Define

\medskip

\begin{minipage}{.4\linewidth}
\begin{equation*}
a_\alpha=
  \left\{\begin{aligned}
  &   x_\alpha &\mbox{if }\alpha<\gamma\\
&0  &\mbox{if }\alpha=\gamma\\
&1  &\mbox{if }\alpha>\gamma
\end{aligned}
 \right.
\end{equation*}
\end{minipage}
\begin{minipage}{.1\linewidth}
 and
\end{minipage}
\begin{minipage}{.3\linewidth}
\begin{equation*}
b_\alpha=
  \left\{\begin{aligned}
  &   x_\alpha &\mbox{if }\alpha<\beta\\
&\varepsilon  &\mbox{if }\alpha=\beta\\
&0  &\mbox{if }\alpha>\beta
\end{aligned}
\right.
\end{equation*}
\end{minipage}

\medskip

Let $W$ be the open subset of $[0,1]^\lambda_{lex}$ consisting of all elements of $[0,1]^\lambda_{lex}$ that lie between $(a_\alpha)_{\alpha<\lambda}$ and $(b_\alpha)_{\alpha<\lambda}$ in the lexicographic order. Clearly, $(x_\alpha)\in W$ and one easily verifies that the set $W\times p^{-1}_\beta(B)$ is an open neighborhood of $\la(x_\alpha), (y_\alpha) \ra$ in $[0,1]^\lambda_{lex}\times [0,1]^\lambda$ disjoint from $F$.

\medskip

\textit{Case 3:} $x_\beta=1$. This case is analogous to the previous one. There is $\alpha>\beta$ with $x_\alpha\neq 1$. Let
$$\gamma=\min\{\alpha>\beta:x_\alpha\neq 1\}.$$
Fix $\varepsilon>0$ such that $(\varepsilon,1]\subseteq A$. Define

\medskip

\begin{minipage}{.4\linewidth}
\begin{equation*}
a_\alpha=
  \left\{\begin{aligned}
  &   x_\alpha &\mbox{if }\alpha<\beta\\
&\varepsilon  &\mbox{if }\alpha=\beta\\
&1  &\mbox{if }\alpha>\beta
\end{aligned}
 \right.
\end{equation*}
\end{minipage}
\begin{minipage}{.1\linewidth}
 and
\end{minipage}
\begin{minipage}{.3\linewidth}
\begin{equation*}
b_\alpha=
  \left\{\begin{aligned}
  &   x_\alpha &\mbox{if }\alpha<\gamma\\
&1  &\mbox{if }\alpha=\gamma\\
&0  &\mbox{if }\alpha>\gamma
\end{aligned}
\right.
\end{equation*}
\end{minipage}

\medskip

As in Case 2, let $W$ consists of all elements in $[0,1]^\lambda_{lex}$ that lie between $(a_\alpha)_{\alpha<\lambda}$ and $(b_\alpha)_{\alpha<\lambda}$. Then, the set $W\times p_\beta^{-1}(B)$ is an open neighborhood of $\la(x_\alpha), (y_\alpha) \ra$ in $[0,1]^\lambda_{lex}\times [0,1]^\lambda$ disjoint from $F$.

\medskip

This finishes the proof of the claim.
\end{proof}
Let $$\pi_1:[0,1]^\lambda_{lex}\times [0,1]^\lambda\to [0,1]^\lambda_{lex} \quad\mbox{and}\quad \pi_2:[0,1]^\lambda_{lex}\times [0,1]^\lambda\to [0,1]^\lambda$$
be projections.
Observe that $\pi_1(F)=[0,1]^\lambda_{lex}$ and if $y\in [0,1]^\lambda$, then the set $\pi^{-1}_2(y)\cap F$ is countable and compact, thus metrizable. It follows from Proposition \ref{KLSigma_characterization_2} that $[0,1]^\lambda_{lex}\in KL\Sigma(\leq \omega)$.
\end{proof}

From Theorem \ref{ExampleD:main}, Proposition \ref{KLSigma_characterization} and \cite[Example 2.4]{Tkachuk} we get the following.

\begin{cor}\label{cor_lex_product}
 The lexicographic product of three intervals $[0,1]^3_{lex}$ is a first-countable continuous image of a metrizably fibered space, yet it is not metrizably fibered.
\end{cor}

\subsection{Example D} Let us describe now the space $KD$ from Example D announced in the Introduction (see Propositions \ref{KD:proposition1} and \ref{KD:proposition2} below). Consider the following lexicographic product:
$$X=[0,1]\times_{lex} [0,1] \times_{lex} \{0,1\}.$$
Let $KD$ be the quotient space obtained from $X$ by identifying each pair of points $(t,0,0)$ and $(t,1,1)$, where $t\in [0,1]$. Let $q:X\to KD$ be the quotient map. It is clear that $KD$ is a compact first-countable space.

The following fact which asserts that $KD$ is not metrizably fibered can be proved essentially in the same way as \cite[Example 2.4]{Tkachuk}. We enclose the argument for the convenience of the reader.

\begin{prop}\label{KD:proposition1}
The space $KD$ is not metrizably fibered
\end{prop}
\begin{proof} For $t\in [0,1]$ consider the following subset $X_t$ of $X$
$$X_t=\{t\}\times[0,1]\times\{0,1\}.$$
Let $M$ be a compact metric space with a metric $d$. Fix a continuous map $f:KD\to M$. Since each $q(X_t)$ is nonmetrizable (because $q(X_t)$ contains a copy of the split interval), it is enough to show that the set
$A=\{t\in [0,1]:|f(q(X_t))|>1\}$
is countable. Striving for a contradiction, suppose that $A$ is uncountable. For each $t\in A$, fix $a_t,b_t\in X_t$ with $f(q(a_t))\neq f(q(b_t))$. As $A$ is uncountable, there is $\varepsilon>0$ and an uncountable $B\subseteq A$ such that $d(q(f(a_t)),f(q(b_t)))\geq \varepsilon$. The set $B$ being an uncountable subset of $[0,1]$ contains a nontrivial sequence $S=\{t_n:n\in \omega\}\subseteq B$ convergent to a point $p\in B$.
Without loss of generality we may assume that $S$ is monotone (as $S$ contains a monotone subsequence). But then both $(a_{t_n})_{n\in \omega}$ and $(b_{t_n})_{n\in \omega}$ converge in $X$ to the same point: either to $(p,0,0)$ if $S$ is increasing, or to $(p,1,1)$ if $S$ is decreasing. This is a contradiction because $d(q(f(a_{t_n})),f(q(b_{t_n})))\geq \varepsilon$, for all $n\in \omega$.
\end{proof}

\begin{prop}\label{KD:proposition2}
The space $KD$ is a continuous image of a $3$-fibered compactum.
\end{prop}
\begin{proof}
 Consider the following subspace $F$ of the Cartesian product $KD\times [0,1]^2$
 \begin{align*}
  F=&\{\la q(x,y,i), (x,y)\ra: (x,y)\in [0,1]^2,\;i\in\{0,1\}\}\\
  &\cup\{\la q(x,0,0), (x,y)\ra: (x,y)\in [0,1]^2 \}
 \end{align*}
\begin{claim}
 The set $F$ is closed in $KD\times [0,1]^2$.
\end{claim}
\begin{proof}
Pick $\la z, (x,y)\ra\in (KD\times [0,1]^2)\setminus F$. Take $a,b\in [0,1]$ and $i\in \{0,1\}$ such that $z=q(a,b,i)$. Since $\la z, (x,y)\ra\notin F$, one of the following three cases holds:

\medskip

\textit{Case 1:} $a\neq x$. Let $A$ and $B$ be two disjoint open subsets of $[0,1]$ satisfying $a\in A$ and $x\in B$. We put
$$U=q(A\times [0,1]\times \{0,1\}) \quad \mbox{and}\quad V=B\times [0,1].$$
The set $U$ is open in $KD$, as $q^{-1}(U)=A\times [0,1]\times \{0,1\}$ and the latter set is open in $X$. Therefore, the set $U\times V$ is an open neighborhood of $\la z,(x,y) \ra$ disjoint from $F$.

\medskip

\textit{Case 2:} $a=x$ and $0<b<1$. Since $\la z,(x,y) \ra\notin F$ and $z=q(a,b,i)=q(x,b,i)$, we have $b\neq y$. Let $A$ and $B$ be two disjoint open sets in $[0,1]$ satisfying $b\in A$ and $y\in B$. Since $0<b<1$, there are $\alpha<\beta\in [0,1]$ with $b\in (\alpha,\beta)\subseteq A$. Define $W$ to be the set of all points in $[0,1]\times[0,1]\times\{0,1\}$ that lie strictly between $(x,\alpha,0)$ and $(x,\beta,1)$ in the lexicographic order. The set $W$ is open in $X$ and since $W=q^{-1}(q(W))$, the set $q(W)$ is open in $KD$. Since $\alpha<b<\beta$ we have $(x,b,i)\in W$, whence $z\in q(W)$. Let $V=[0,1]\times B$. We conclude that the set
$q(W)\times V$ is an open neighborhood of $\la z,(x,y) \ra$ in $KD\times [0,1]^2$ disjoint from $F$.

\medskip

\textit{Case 3:} $a=x$, $b\in\{0,1\}$ and $(b,i)\neq (0,0),(1,1)$. As in the previous case, we have $b\neq y$. Let $A$ and $B$ be two disjoint open sets in $[0,1]$ satisfying $b\in A$ and $y\in B$. Since $A$ is an open neighborhood of $b$ in $[0,1]$, there are $\alpha,\beta\in [0,1]$ with $b\in [\alpha,\beta]\subseteq A$. Define $W$ to be the set of all points in $[0,1]\times[0,1]\times\{0,1\}$ that lie strictly between $(x,\alpha,0)$ and $(x,\beta,1)$ in the lexicographic order. The set $W$ is open in $X$ and since $W=q^{-1}(q(W))$, the set $q(W)$ is open in $KD$. Since $(b,i)\neq (0,0),(1,1)$, we have $(x,b,i)\in W$, whence $z\in q(W)$. Let $V=[0,1]\times B$. We conclude that the set
$q(W)\times V$ is an open neighborhood of $\la z,(x,y) \ra$ in $KD\times [0,1]^2$ disjoint from $F$.

\medskip

The proof of the claim is finished.
\end{proof}
Let
$$\pi_1: KD\times [0,1]^2\to KD \quad \mbox{and}\quad \pi_2:KD\times [0,1]^2\to [0,1]^2$$
be projections. First observe that $\pi_1(F)=KD$, for if $q(x,y,i)\in KD$ then $\la q(x,y,i),(x,y) \ra \in F$. So $F$ maps continuously onto $KD$. Next, observe that
if $p=(x,y)\in [0,1]^2$, then
$$\pi_2^{-1}(p)\cap F=\{\la q(x,y,0), p \ra,\; \la q(x,y,1),p \ra,\;\la q(x,0,0), p\ra\}.$$
So $F$ is 3-fibered.
\end{proof}

\section{Continuous images of $n$-fibered compacta}

The purpose of this section is to show that $3$-fibered cannot be improved to $2$-fibered in Example \ref{Example D}.
According to \cite{AviTod_metricdegree},

\begin{defin} We say that the \textit{open degree} of $K$ does not exceed $n$ and write $odeg(K)\leq n$ if there exists a countable family $\mathcal{O}$ of open sets such that for every $n+1$ points $x_0,\ldots,x_n\in K$ there exists $W_i\in \mathcal{O}$ such that $x_i\in W_i$ for $i=0,\ldots,n$ and $\bigcap_{i=0}^n W_i = \emptyset$.
\end{defin}

The following theorem is an unpublished result of the first author and Todorcevic:

\begin{thrm}\label{odeg_characterization}
	A compact space $K$ satisfies $odeg(K)\leq n$ if and only if $K$ is a continuous image of an $n$-fibered compactum.
\end{thrm}

\begin{proof}
	It is convenient to rephrase the definition of the open degree in terms of closed sets instead of open sets. We have that $odeg(K)\leq n$ if and only if there exists a countable family $\mathcal{F}$ of closed sets such that whenever we are given $n+1$ different points $x_0,\ldots,x_n\in K$ there exist $F_0,\ldots,F_n\in\mathcal{F}$ such that $x_i\not\in F_i$ and $\bigcup_{i=0}^n F_i = K$.

	Suppose first that there are continuous surjections $f:L\to K$ and $g:L\to M$ such that $M$ is metric and $|g^{-1}(x)|\leq n$ for all $x\in M$. Since every compact metric space is a continuous image of the Cantor set, we can find a tree $\mathcal{T} = \{M_s : s\in 2^{<\omega}\}$ of closed subsets of $M$ such that
	\begin{itemize}
		\item $M_\emptyset = M$,
		\item $M_s = M_{s^\frown 0}\cup M_{s^\frown 1}$ for all $s\in 2^{<\omega}$,
		\item $\left|\bigcap_{m<\omega}M_{\sigma|m}\right| = 1$ for all  $\sigma\in 2^{\omega}$.
		\end{itemize}
		Define
		$$\mathcal{F} = \{f(g^{-1}(A)) :A \text{ is a finite union of sets from }\mathcal{T}\}.$$
		Take a set of $n+1$ different points $X=\{x_0,\ldots,x_{n}\}\subset K$. We claim that there exists $m<\omega$ such that $X\not\subset f(g^{-1}(M_s))$ for all $s\in 2^m$. Otherwise, by K\"onig's lemma, there exists $\sigma\in 2^\omega$ such that $X\subset \bigcap_{m<\omega}f(g^{-1}(M_{\sigma|m}))$. But a compactness argument gives that
		$$\bigcap_{m<\omega}f(g^{-1}(M_{\sigma|m})) = f\left(\bigcap_{m<\omega}g^{-1}(M_{\sigma|m})\right) = f\left(g^{-1}\left(\bigcap_{m<\omega}M_{\sigma|m}\right) \right),$$
		that has cardinality at most $n$, so it cannot contain $X$. Now, define
		$$A_i = \bigcup\{M_s : s\in 2^m, x_i\not\in f(g^{-1}(M_s))\}.$$
		The sets $F_i = f(g^{-1}(A_i))$ are as desired.

		Now suppose that $\mathcal{F}$ is a family of closed sets that witnesses that $odeg(K)\leq n$. Consider the countable set
		$$Z = \{F=(F_0,\ldots,F_n) : F_i\in\mathcal{F} \mbox{ and } F_0\cup\cdots\cup F_n= K\}. $$
		Let
		$$ L = \left\{(x,(i_{F})_{F\in Z}) \in K \times \{0,\ldots,n\}^Z : x\in \bigcap_{F\in Z}F_{i_F}\right\}.$$
		The set $L$ is closed, therefore compact. The projection on the first coordinate $L\to K$ is onto. The fact that  $\mathcal{F}$ witnesses that $odeg(K)\leq n$ implies that the projection on the second coordinate $L\to \{0,\ldots,n\}^Z$ is $n$-to-one.
	\end{proof}

\begin{lemma}\label{finiteremovelemma}
	Let $K$ be a compact space and $F\subset K$ be a finite $G_\delta$-subset of $K$. Suppose also that $odeg(L)<n$ for all closed $L\subset K\setminus F$. Then $odeg(K)<n$
\end{lemma}

\begin{proof}
	Consider $$V_0\subset \overline{V_0} \subset V_1 \subset \overline{V_1}\subset \cdots $$ a chain of open sets such that $\bigcup V_i = K\setminus F$.
	For every $m$ take a countable family $\mathcal{O}_m$ of open subsets of $V_{m+1}$ whose intersections with $\overline{V_m}$ witness the fact that $odeg(\overline{V_m})<n$. Put $$\widetilde{\mathcal{O}}_m=\{A\cap V_m:A\in \mathcal{O}_m\}.$$
For every $x\in F$, take a countable basis $\mathcal{N}_x$ of neighborhoods of $x$. The family $\bigcup_m \widetilde{\mathcal{O}}_m \cup \bigcup_{x\in F}\mathcal{N}_x$ witnesses that $odeg(K)<n$.
\end{proof}

\begin{thrm}
	If $K$ is a first-countable compact space and $odeg(K)\leq n$, then there is a continuous surjection $\phi:K\to Z$ onto a metrizable space $Z$  such that $odeg(\phi^{-1}(t))<n$ for all $t\in Z$.
\end{thrm}

\begin{proof}
By Theorem \ref{odeg_characterization}, the space $K$ is a continuous image of an $n$-fibered compact space $L$.
Let $f:L\to K$ be a continuous surjection and let $\pi:L\to M$ a continuous map onto a metric space such that $|\pi^{-1}(z)|\leq n$ for all $z\in M$.

We consider a countable basis $\mathcal{B}$ of $M$. Let
$$ \Gamma = \{(A,B) \in \mathcal{B}\times \mathcal{B} : f(\pi^{-1}(\overline{A}))\cap f(\pi^{-1}(\overline{B})) = \emptyset \}.$$

For each $(A,B)\in \Gamma$, we consider a continuous function $\phi_{(A,B)}:K\to [0,1]$ such that
$\phi_{(A,B)}(f(\pi^{-1}(\overline{A}))) = \{0\}$ and $\phi_{(A,B)}(f(\pi^{-1}(\overline{B})))=\{1\}$. We then take
$\phi:K\to [0,1]^\Gamma$ be given by $\phi(x)_{(A,B)} = \phi_{(A,B)}(x)$. Since $\Gamma$ is countable, $[0,1]^\Gamma$ is metrizable. So it is enough to prove that $odeg(\phi^{-1}(t))<n$ for every $t\in [0,1]^\Gamma$.

\begin{claim}
For every $x,y\in L$, if $\phi(f(x)) = \phi(f(y))$, then there exist $\tilde{x},\tilde{y}\in L$ such that $\pi(\tilde{x}) = \pi(x)$, $\pi(\tilde{y}) = \pi(y)$ and $f(\tilde{x})=f(\tilde{y})$.
\end{claim}

\begin{proof}
For every basic neighborhoods $A$ of $\pi(x)$ and $B$ of $\pi(y)$ we must have that $f(\pi^{-1}(\overline{A}))\cap f(\pi^{-1}(\overline{B}))\neq \emptyset$ because otherwise $\phi_{(A,B)}(x)=0$ and $\phi_{(A,B)}(y)=1$ which contradicts that $\phi(f(x)) = \phi(f(y))$. Take $\{A_m\}_{m\in\mathbb{N}}$ and $\{B_m\}_{m\in\mathbb{N}}$ decreasing countable bases of neighborhoods of $\pi(x)$ and $\pi(y)$ respectively. For every $m$ we will be able to find $\tilde{x}_m \in \pi^{-1}(\overline{A}_m)$ and $\tilde{y}_m\in \pi^{-1}(\overline{B}_m)$ such that $f(\tilde{x}_m) = f(\tilde{y}_m)$. If we now take a nonprincipal ultrafilter $\mathcal{U}$, then $\tilde{x} = \lim_\mathcal{U}\tilde{x}_m$ and $\tilde{y} = \lim_\mathcal{U}\tilde{y}_m$ are as desired.
\end{proof}

Now fix $t\in [0,1]^\Gamma$ with $\phi^{-1}(t)\neq\emptyset$, fix $y\in L$ such that $\phi(f(y)) = t$ and $s=\pi(y)$. Let $F = f(\pi^{-1}(s))$. By Lemma~\ref{finiteremovelemma}, it is enough to prove that $odeg(A)<n$ whenever $A\subset \phi^{-1}(t)\setminus F$ is closed. For this, it is enough to prove that $\pi: f^{-1}(A) \to M$ is at most $(n-1)$-to-one (cf. Theorem \ref{odeg_characterization}). If it was not the case, then there must exist $r\in M$ such that $\pi^{-1}(r) \subset f^{-1}(A)$. But given $x\in \pi^{-1}(r)$, notice that $\phi(f(x)) = \phi(f(y))$, so we can consider the $\tilde{x}$ and $\tilde{y}$ provided by the Claim above. Notice that $\tilde{x}\in \pi^{-1}(r)$ since $\pi(\tilde{x}) = \pi(x)$. Also, $f(\tilde{x}) = f(\tilde{y})\not\in A$ because $\pi(\tilde{y})=\pi(y) = s$, so $f(\tilde{y})\in F$. So we found $\tilde{x}\in \pi^{-1}(r) \setminus f^{-1}(A)$ a contradiction.

\end{proof}

\begin{cor}\label{theorem1.3}
	If a first-countable compact $K$ is a continuous image of a 2-fibered compactum, then $K$ is metrizably fibered.
\end{cor}

\begin{proof}
Apply the previous theorem for $n=1$.
\end{proof}

\section{Example E: A first-countable compact $L\Sigma(\leq\omega)$-space which is not a continuous image of any metrizably fibered compactum}

Let $Y=[0,1]\times \mathbb{S}$ be the lexicographic product of $[0,1]$ and the split interval $\mathbb{S}$.
By $I(\la x,s\ra ,\la y,t\ra)$ we denote the set of all points lying strictly between $\la x,s \ra$ and $\la y,t\ra$ in the lexicographic order.

Let $X=[0,1]\times(0,1]$.
For a positive integer $n$ and $z,a,b\in [0,1]$ such that $a<b$ we set (see Figure \ref{figure5} below)
$$U_n(z,a,b)=X\cap\left\{\la z+r\cos(\pi-\pi t), r\sin(\pi-\pi t)\ra\in \RR^2:0<r<\tfrac{1}{n},\;a<t<b\right\}.$$

\begin{figure}[H]
 \caption{The set $U_n(z,a,b)$ consists of the shaded area without the boundary. The angle marked in blue has measure $\pi a$; the angle marked in red has measure $\pi b$.}
 \label{figure5}

 \medskip

 \includegraphics{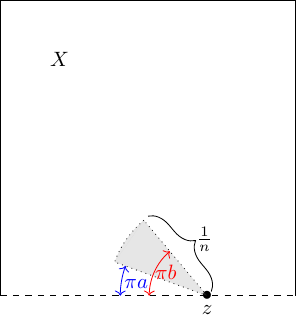}
\end{figure}

We define the following topology on $KE=X\cup Y$:
If $p\in X$ then $U$ is a neighborhood of $p$ in $KE$ if $U$ is a neighborhood of $p$ in $X$ (in its usual product topology). To define neighborhoods of points $p\in Y$ we shall consider the following four cases:

\textit{Case 1:} $p=\la x,y^-\ra\in Y=[0,1]\times \mathbb{S}$ and $y\notin\{0,1\}$.
In this case basic open neighborhoods of $p$ are of the form
$$U_n(x,a,y)\cup I(\la x,a^-\ra, \la x,y^+\ra), \mbox{ where } a<y.$$

\textit{Case 2:} $p=\la x,y^+\ra\in Y=[0,1]\times \mathbb{S}$ and $y\notin\{0,1\}$.
Then basic open neighborhoods of $p$ are of the form
$$U_n(x,y,b)\cup I(\la x,y^-\ra,\la x,b^+\ra), \mbox{ where } y<b.$$

\textit{Case 3:} $p=\la x,0^+\ra\in Y=[0,1]\times \mathbb{S}$
Then basic open neighborhoods of $p$ are of the form
$$U_n(x,0,b)\cup I(\la x-\tfrac{1}{n},\tfrac{1}{2}^-\ra,\la x,b^+\ra), \mbox{ where } b>0.$$

\textit{Case 4:} $p=\la x,1^-\ra\in Y=[0,1]\times \mathbb{S}$
Then basic open neighborhoods of $p$ are of the form
$$U_n(x,a,1)\cup I(\la x,a^-\ra,\la x+\tfrac{1}{n},\tfrac{1}{2}^+\ra), \mbox{ where } a<1.$$

One can readily check that this is a well defined neighborhood system on $KE$ and that the topology defined in this way is Hausdorff, separable and first-countable. It is also evident that the subspace topology on $X$ (respectively, $Y$) agrees with the product topology on $X$ (the lexicographic order topology on $Y$). To show that $KE$ is compact, fix an open cover $\mathcal{U}$ of $KE$. Since $Y$ is a compact subset of $KE$, being the lexicographic product of $[0,1]$ and $\mathbb{S}$, there is a finite subfamily $\mathcal{V}$ of $\mathcal{U}$ that covers $Y$. Now, $KE\setminus \bigcup\mathcal{V} \subseteq X$ is closed in $KE$ and hence it is closed in the product topology of $[0,1]\times [0,1]$. It follows that there is a finite subfamily $\mathcal{V}'\subseteq \mathcal{U}$ that covers $KE\setminus \bigcup\mathcal{V}$. Thus, $\mathcal{V}\cup\mathcal{V}'$ is a finite subcover of $\mathcal{U}$.

\begin{lemma}\label{lemma_dichotomy}
 Let $T\subseteq [0,1]$ be an open interval and let $a>0$. Suppose that $C_0,C_1$ are closed subsets of $KE$ covering the set
 $(T\times \mathbb{S})\cup (T\times (0,a))$.
 Then either
 \begin{enumerate}
  \item The set $X(C_0)=\{x\in T: \exists s\in \mathbb{S}\;\; \la x,s \ra \notin C_0 \}$ is meager in $T$ or
  \item There is an open interval $J\subseteq T$ and $\varepsilon>0$ such that
  $(J\times \mathbb{S})\cup (J\times (0,\varepsilon)) \subseteq C_1$.
 \end{enumerate}
\end{lemma}
\begin{proof}
Suppose that $X(C_0)$ is non-meager in $T$. We will show that assertion (2) holds. The set $C_0$ is closed in $KE$, so if $x\in X(C_0)$, then for some $s\in \mathbb{S}$, there is a basic open neighborhood $U$ of $\la x,s \ra$ disjoint from $C_0$. Hence, for some positive integer $n$ and rationals $q<r$ we have
 $$U_n(x,q,r)\cap C_0=\emptyset.$$
 We can therefore write $$X(C_0)=\bigcup\{X(n,q,r):n\geq 1, q,r\in \mathbb{Q}\cap[0,1], q<r\},$$ where
 $X(n,q,r)=\{x\in T: U_n(x,q,r)\cap C_0=\emptyset\}$.
 Since $X(C_0)$ is not meager in $T$, we can find $n,q,r$ so that
 $X(n,q,r)$ is not nowhere dense in $T$, i.e. $X(n,q,r)$ is dense in some open interval $J'\subseteq T$. It follows that for sufficiently small $\varepsilon>0$ and some interval $J\subseteq J'$, the family $\{U_n(x,q,r):x\in X(n,q,r))\}$ covers the whole rectangle $J\times (0,\varepsilon)$, whence $J\times (0,\varepsilon)$ is disjoint from $C_0$. Since $C_0\cup C_1=KE$, we conclude that $J\times (0,\varepsilon)\subseteq C_1$.
 Observe that if $\la x ,s\ra\in J\times \mathbb{S}$, then every open neighborhood of $\la x ,s\ra$ in $KE$ meets $J\times (0,\varepsilon)$. Consequently,  $J\times \mathbb{S}\subseteq C_1$ because $C_1$ is closed. This gives assertion (2).
\end{proof}

\begin{prop}\label{proposition_comeager}
 Let $I\subseteq [0,1]$ be an open interval.
 Suppose that $\mathcal{C}$ is a finite family of closed subsets of $KE$.
 If $\mathcal{C}$ covers a set $(T\times \mathbb{S}) \cup (T\times (0,a))$, for some open interval $T\subseteq I$ and some $a>0$, then there is $C\in \mathcal{C}$ and an open interval
 $J\subseteq I$ such that the set $\{x\in J:\forall s\in \mathbb{S}\;\;\la x,s \ra\in C\}$ is comeager in $J$.
 \end{prop}
\begin{proof}
 We proceed by induction on the size of $\mathcal{C}$. If $\mathcal{C}$ consists of one set and covers  $(T\times \mathbb{S}) \cup (T\times (0,a))$, then we can obviously take $J=T$.
 Fix $n\geq 1$ and suppose that our assertion holds for any family $\mathcal{C}$ of size $n$. Let $\mathcal{C}'$ be a family of size $n+1$.
 Suppose that for some open interval $T\subseteq I$ and some $a>0$, we have
 \begin{equation}\label{equation_cover}
  (T\times \mathbb{S}) \cup (T\times (0,a))\subseteq \bigcup \mathcal{C}'
 \end{equation}
 Pick $F\in \mathcal{C}'$ and define $\mathcal{C}=\mathcal{C}'\setminus \{F\}$.
 By \eqref{equation_cover}, we may apply Lemma \ref{lemma_dichotomy}, to $C_0=F$ and $C_1=\bigcup\mathcal{C}$. If $X(C_0)$ is meager in $T$, then we take $J=T$ and $C=C_0=F$. Otherwise, by Lemma \ref{lemma_dichotomy}, there is an open interval $J'\subseteq T$ and $\varepsilon >0$ such that
 $$(J'\times \mathbb{S})\cup (J'\times (0,\varepsilon)) \subseteq C_1=\bigcup C.$$
 Since $|\mathcal{C}|=n$, the result follows from the inductive assumption.
 \end{proof}
 \begin{cor}\label{Corollary_comeager}
  If $\mathcal{C}$ is a finite cover of $KE$ consisting of closed sets, then the set
  $$X(\mathcal{C})=\bigcup_{C\in \mathcal{C}}\{x\in [0,1]:\forall s\in \mathbb{S}\;\;\la x,s \ra\in C\}$$
  is comeager in $[0,1]$.
 \end{cor}
 \begin{proof}
  By Proposition \ref{proposition_comeager}, for every open interval $I\subseteq [0,1]$, we can find an open interval $J\subseteq I$ so that the set
  $X(\mathcal{C})\cap J$ is comeager in $J$. Hence, $X(\mathcal{C})$ is comeager in $[0,1]$ (cf. \cite[8.29]{Ke}).
 \end{proof}

 \begin{thrm}\label{ExampleE is not KLS}
The space $KE$ is not a continuous image of any metrizably fibered compactum.
\end{thrm}
\begin{proof}
Let $\mathcal{C}=\{C_t:t\in 2^{< \omega}\}$ be an arbitrary family of closed subsets of $KE$ such that
$C_\emptyset=KE$ and $C_t=C_{t\frown 0}\cup C_{t\frown 1}$, for every $t\in 2^{<\omega}$.
According to Proposition \ref{KLSigma_characterization}, it is enough to find $\sigma\in 2^\omega$ so that $\bigcap_n C_{\sigma|n}$ is nonmetrizable.

For $n\in \omega$ denote $\mathcal{C}_n=\{C_t:|t|=n\}$. By Corollary \ref{Corollary_comeager}, for every $n\in \omega$, the set
$$X(\mathcal{C}_n)=\bigcup_{C\in \mathcal{C}_n}\{x\in [0,1]:\forall s\in \mathbb{S}\;\;\la x,s \ra\in C\}$$
is comeager in $[0,1]$, whence $\bigcap_{n\in \omega} X(\mathcal{C}_n)\neq\emptyset$. Pick $a\in \bigcap_{n\in \omega} X(\mathcal{C}_n)$. For every $n$ there is $t\in 2^{<\omega}$ of length $n$ such that for every $s\in \mathbb{S}$ we have $\la a,s \ra\in C_t$. It follows that the set
$$\{t\in 2^{<\omega}:\forall s\in \mathbb{S}\;\;\la a,s \ra\in C_t\}$$ is an infinite tree and thus by K\"onig's lemma it has an infinite branch $\sigma$. Now, the set $\bigcap_{n\in\omega}C_{\sigma|n}$ is nonmetrizable because it contains
a copy of the split interval $\{a\}\times \mathbb{S}$.
\end{proof}

The corollary below gives a partial answer to a question of Tkachuk \cite[Problem 3.4]{Tkachuk}.

\begin{cor}
The space $KE$ is a compact first-countable $L\Sigma(\leq\omega)$-space which is not a continuous image of any compact metrizably fibered space.
\end{cor}
\begin{proof}
By definition of the topology, the space $KE$ is compact and first-countable. Recall that $X=[0,1]\times (0,1]$ and $Y=[0,1]\times \mathbb{S}$. Both $X$ and $Y$ are $L\Sigma(\leq\omega)$-spaces (cf. \cite[Proposition 2.5]{KOS}) so $KE$ is an $L\Sigma(\leq\omega)$-space being the union of $X$ and $Y$ (see \cite[Proposition 2.4]{KOS}). According to Theorem \ref{ExampleE is not KLS}, the space $KE$ is as required.
\end{proof}

\section*{Acknowledgements}
The authors were partially supported by
Fundaci\'{o}n S\'{e}neca - ACyT Regi\'{o}n de Murcia project 21955/PI/22, Agencia Estatal de Investigación (Government of Spain) and ERDF project PID2021-122126NB-C32 (A. Avil\'es and M. Krupski); European Union - NextGenerationEU funds through Mar\'{i}a Zambrano fellowship and
the NCN (National Science Centre, Poland) research Grant no. 2020/37/B/ST1/02613 (M. Krupski)

\bibliographystyle{siam}
\bibliography{bib.bib}

\begin{thebibliography}{10}

\bibitem{AviTod_metricdegree}
{\sc A.~Avil\'{e}s and S.~Todorcevic}, {\em Compact spaces of the first {B}aire
  class that have open finite degree}, J. Inst. Math. Jussieu, 17 (2018),
  pp.~1173--1196.

\bibitem{AT}
\leavevmode\vrule height 2pt depth -1.6pt width 23pt, {\em Lexicographic
  products as compact spaces of the first {B}aire class}, Topology Appl., 267
  (2019), pp.~106871, 7.

\bibitem{D}
{\sc D.~Daniel}, {\em Metrizable fiberings of continuous {H}ausdorff images of
  compact ordered spaces}, Topology Proc., 31 (2007), pp.~77--87.

\bibitem{DK}
{\sc D.~Daniel and C.~T. Kennaugh}, {\em Metrizable fiberings of continuous
  images of the lexicographic square}, Topology Appl., 153 (2006),
  pp.~1603--1608.

\bibitem{Eng}
{\sc R.~Engelking}, {\em General topology}, vol.~6 of Sigma Series in Pure
  Mathematics, Heldermann Verlag, Berlin, second~ed., 1989.

\bibitem{Ke}
{\sc A.~S. Kechris}, {\em Classical descriptive set theory}, vol.~156 of
  Graduate Texts in Mathematics, Springer-Verlag, New York, 1995.

\bibitem{KM}
{\sc W.~Kubi\'{s} and A.~Molt\'{o}}, {\em Finitely fibered {R}osenthal compacta
  and trees}, Rev. R. Acad. Cienc. Exactas F\'{\i}s. Nat. Ser. A Mat. RACSAM,
  105 (2011), pp.~23--37.

\bibitem{KOS}
{\sc W.~Kubi\'{s}, O.~Okunev, and P.~J. Szeptycki}, {\em On some classes of
  {L}indel\"{o}f {$\Sigma$}-spaces}, Topology Appl., 153 (2006),
  pp.~2574--2590.

\bibitem{L}
{\sc A.~Leiderman}, {\em Trees and some classes of {C}orson compacts}.
\newblock Abstract. Intalian-Spanish Conference on Topology and its
  Applications, Trieste, 2012.

\bibitem{L1}
\leavevmode\vrule height 2pt depth -1.6pt width 23pt, {\em Everywhere dense
  metrizable subspaces of {C}orson compacta}, Mat. Zametki, 38 (1985),
  pp.~440--449.

\bibitem{MP}
{\sc W.~Marciszewski and R.~Pol}, {\em On some problems concerning {B}orel
  structures in function spaces}, Rev. R. Acad. Cienc. Exactas F\'{\i}s. Nat.
  Ser. A Mat. RACSAM, 104 (2010), pp.~327--335.

\bibitem{MLO}
{\sc I.~Molina~Lara and O.~Okunev}, {\em {$L\Sigma(\leq\omega)$}-spaces and
  spaces of continuous functions}, Cent. Eur. J. Math., 8 (2010), pp.~754--762.

\bibitem{O}
{\sc O.~Okunev}, {\em {$L\Sigma(\kappa)$}-spaces}, in Open Problems in Topology
  II, E.~Pearl, ed., Elsevier, 2007.

\bibitem{Pol}
{\sc R.~Pol}, {\em Note on compact sets of first {B}aire class functions},
  Proc. Amer. Math. Soc., 96 (1986), pp.~152--154.

\bibitem{RH}
{\sc R.~Rojas-Hern\'{a}ndez}, {\em {$\Sigma_s$}-products revisited}, Comment.
  Math. Univ. Carolin., 56 (2015), pp.~243--255.

\bibitem{S}
{\sc G.~A. Sokolov}, {\em On some classes of compact spaces lying in
  {$\Sigma$}-products}, Comment. Math. Univ. Carolin., 25 (1984), pp.~219--231.

\bibitem{Tkachuk}
{\sc V.~V. Tkachuk}, {\em A glance at compact spaces which map ``nicely'' onto
  the metrizable ones}, Topology Proc., 19 (1994), pp.~321--334.

\bibitem{Tkachuk1}
\leavevmode\vrule height 2pt depth -1.6pt width 23pt, {\em Lindel\"{o}f
  {$\Sigma$}-spaces: an omnipresent class}, Rev. R. Acad. Cienc. Exactas
  F\'{\i}s. Nat. Ser. A Mat. RACSAM, 104 (2010), pp.~221--244.

\bibitem{Tkachenko}
{\sc M.~G. Tka\v{c}enko}, {\em $\mathcal{P}$-approximable compact spaces},
  Comment. Math. Univ. Carolin., 32 (1991), pp.~583--595.

\bibitem{Tod1}
{\sc S.~Todorcevic}, {\em Dense metrizability}.
\newblock to appear in Ann. Pure Appl. Log.

\bibitem{Tod}
\leavevmode\vrule height 2pt depth -1.6pt width 23pt, {\em Topics in Topology},
  vol.~1652 of Lecture Notes in Mathematics, Springer-Verlag,
  Berlin--Heidelberg, second~ed., 1997.

\end{thebibliography}
\end{document}